\documentclass{amsart}
\usepackage{amsmath,eucal,graphicx,amsxtra,hyperref}
\usepackage{amsrefs}

\input xy
\xyoption{all}

\newtheorem{theorem}{Theorem}[section]
\newtheorem{lemma}[theorem]{Lemma}

\theoremstyle{definition}
\newtheorem{definition}[theorem]{Definition}
\newtheorem{example}[theorem]{Example}

\theoremstyle{remark}
\newtheorem{remark}[theorem]{Remark}

\newcommand{\ga}{{\bf g}}
\newcommand{\Cyc}{{\mathcal C}}

\newcommand{\C}{\mathbb{C}}

\newcommand{\Z}{\mathbb{Z}}
\newcommand{\N}{\mathbb{N}}

\renewcommand{\P}{\mathbb{P}}
\newcommand{\PP}{\mathbb{P}}
\newcommand{\GG}{\mathbb{G}}

\begin{document}
\title[Polar maps for Hypersurfaces]{Images of the Polar maps for Hypersurfaces}
\date{\today}
\author[L. \,E. Lopez]{Luis E. Lopez}
\address{Max-Planck-Institut f\"ur Mathematik\\
         Vivatsgasse 7\\
	 D-53111 Bonn \\
	 Germany
}
\curraddr{Max-Planck-Institut f\"ur Mathematik\\
         Vivatsgasse 7\\
	 D-53111 Bonn \\
	 Germany
}
\email{llopez@mpim-bonn.mpg.de}
\thanks{This work is based on part of the author's dissertation. It is a pleasure to thank H. Blaine Lawson for all
his guidance and helpful remarks.}
\begin{abstract} For a projective hypersurface $X \subset \P^n$, the images of the polar maps of 
degree $k$ are studied. The cohomology class defined by these
maps is calculated and classical results on dual varieties are presented as applications.
\end{abstract}
\maketitle


\section{Introduction}
For any smooth variety $X \subset \PP^n$ of 
codimension $q$ the classical  gauss map is the map 
$$g^1:X \rightarrow \GG^q(\PP^n)$$ 
which associates to each point $\xi$
the projective linear subspace of codimension $q$ tangent to $X$ at $\xi$ in $\PP^n$:
$$\xi \mapsto \overline{T_{\xi} X}$$

If $X$ is a hypersurface defined by the set of zeros of a homogeneous polynomial $F$ of degree $d$, 
$X=V(F):=\{x \in \PP^{n} \mid F(x)=0\}  \subset \PP^{n}$, 
then the gauss map has the following coordinate expression
\begin{equation}\label{eq:one}
\xi \mapsto V \left( \sum \frac{\partial F}{\partial x_i}(\xi) {x}_i \right) 
\end{equation}

If $X$ has singularities we no longer have a map which is regular, but only a rational
map. \\

We may think of the Gauss map in the following way: For a fixed point $p \in X$ we associate
an algebraic cycle  in  $\P^n$ of degree one which is obtained via the first partial derivatives
of the polynomial defining $X$. If we take higher derivatives of this polynomial, then it is possible to associate
to every point $p \in X$ an algebraic cycle of higher degree. For a fixed point $p$, this cycle has
appeared in different contexts, we will mention two of them

\subsection{Polar Transformations and Homaloidal Polynomials}

The map  (\ref{eq:one}) has a natural generalization taking higher derivatives.
The following definition (c.f. ~\cite{dolga}, ~\cite{cili}) suggests such a generalization.

\begin{definition} Let $\mathbf{p}=(p_0,\ldots,p_n) \in \C^{n+1} \setminus \{0\}$, and let
$p=[p_0:\ldots:p_n]$ be the corresponding point in $\P^n$. For every positive integer $k<d$
the {\em $k$-th polar polynomial of $F$} is given by
\begin{equation}
\Delta_{\mathbf{p}}^s(F)(x):= \left(p_0\frac{\partial}{\partial x_0}+\cdots+p_n\frac{\partial}{\partial x_n}\right)^{(s)}F(x)
\end{equation}
where $(s)$ denotes the symbolic power taking derivatives and products.
\end{definition}

With this notation, the {\em reciprocity theorem} proved in ~\cite{dolga} implies that 
the gauss map is given by
$$\xi \mapsto V\left(\Delta_{\xi}^{(d-1)}(F)(x)\right).$$

Furthermore, the image of the (rational) Gauss map is the {\em dual} (or polar) variety of $X$, an object which has
been thoroughly studied in classical algebraic geometry.

\subsection{Higher Fundamental Forms}

In ~\cite{behesh} (c.f. also ~\cite{GH} and ~\cite{Lands}), Beheshti  defined for a smooth point $p \in X$ the hypersurfaces $Y^k_p$ 
as the zero set of the polynomial

\begin{equation}
\sum_{0 \leq i_1,\ldots,i_k \leq n} \frac{\partial^k P}{\partial x_{i_1}\cdots \partial x_{i_k}}(p)x_{i_1}\cdots x_{i_k}
\end{equation}

The hypersurface $Y^1_p$ is the tangent plane at $p$ and the restriction of the hypersurface $Y^2_p$ to $T_p$ is a quadratic
form: the second fundamental form.\\

In this paper we will study the properties of the map
$$g^k: X \rightarrow \Cyc^1_k(\PP^n)$$
which associates to each point $\xi$ the effective algebraic cycle of degree $k$ and codimension
$1$ which approximates $X$ at $\xi$. \\

As in the case of the Gauss map, the higher degree polar  maps are only rational in general. Interestingly
however,  they can still be regular in the presence of certain singularities. 

\section{Basic Properties}

Let us recall the Euler relation:
\begin{equation}
\label{euler}
d \cdot F = \sum_{i=0}^n \frac{\partial F}{\partial x_i} x_i
\end{equation}

Iterating the  Euler relation we get  the following:

\begin{equation}
\label{euler1}
(d-1) \cdot \frac{\partial F}{\partial x_i} = \sum_{k=0}^n 
\frac{\partial^2 F}{\partial x_k \partial x_i} x_k
\end{equation}

If we substitute equation (\ref{euler1}) into equation (\ref{euler}) we get the following
relation for $F$:
\begin{equation}
d \cdot (d-1) \cdot F = \sum_{i=0}^n  \sum_{k=0}^n \frac{\partial^2 F}{\partial x_k \partial x_i} x_k x_i
\end{equation}

In general, if $ s \leq d=\deg(F)$ we have the following equation:
\begin{equation}
\boxed{d(d-1)\cdots (d-s+1)F =  \sum_{|\alpha|=s} \frac{\partial^{|\alpha|}F}{\partial x^{\alpha}}  x^{\alpha}}
\end{equation}
where $\alpha$ runs over all multi-indices of length $s$, i.e. 
$$ \alpha =(\alpha_0,\ldots,\alpha_n) \: \text{with} \: \alpha_i \in \N$$
and 
$$|\alpha|=\alpha_0 +\cdots+\alpha_n$$

$$
\frac{\partial^{|\alpha|}F}{\partial x^{\alpha}} = \frac{\partial^{|\alpha|}F}{\partial x_{\alpha_0} \partial x_{\alpha_1} \cdots \partial x_{\alpha_n}}
$$

\begin{remark} One of the consequences of the Euler formula is that the systems
$\{\frac{\partial F}{\partial x_0},\ldots,\frac{\partial F}{\partial x_n},F\}$
and
$\{\frac{\partial F}{\partial x_0},\ldots,\frac{\partial F}{\partial x_n}\}$
have the same set of solutions, more precisely, they define the same scheme since the
ideals they generate are equal. Recursively we
obtain the following lemma.
\end{remark}
\begin{lemma} \label{workhorse} Let $\xi$ be a point in $\P^n$ such that
$\frac{\partial^s F}{\partial x_{\alpha}}(\xi)=0$ for all $\alpha$ with
$|\alpha|=s$. Then $F(\xi)=0$ and $\frac{\partial^{|\beta|} F}{\partial
x^{\beta}}(\xi)=0$ for all $\beta$ with $|\beta| \leq s$ 
\end{lemma}

Now we define the higher degree polar  maps and we derive some consequences
from the generalized Euler formulas given above.

\begin{definition} For every $k \leq d$, the {\em degree $k$ polar map} 
\begin{equation}
\xymatrix{
g^k: X \ar@{-->}[r] & \Cyc^1_k(\P^n) 
}
\end{equation}
is the rational map defined by
$$
\xi \mapsto V\left(\sum_{|\alpha|=k} \frac{\partial^k F}{\partial x_{\alpha}}(\xi)x^{\alpha} \right)
$$
The space $\Cyc^1_k(\P^n)$ of cycles of codimension $1$ and degree $k$ in $\P^n$ can be identified
with $\P^{{n+k \choose k} -1}$, this identification is via the Chow coordinates. Every codimension $1$ 
cycle is determined by a multivariable homogeneous polynomial of degree $k$. If a cycle is 
defined by a polynomial $\sum a_{\alpha} x^{\alpha}$ then its Chow coordinates are
$\left[a_0:\cdots: a_{\alpha}:\cdots \right]$. Using the Chow coordinates in $\Cyc^1_k(\P^n) 
\cong \P^{{n+k \choose k} -1}$ the degree $k$ Gauss map is just 
$$
\xi \mapsto \left[\frac{\partial^k F}{\partial x_0 \cdots \partial x_0}(\xi):\ldots
:\frac{\partial^k F}{\partial x^{\alpha}}(\xi):\ldots:\frac{\partial^k F}
{\partial x_n \cdots \partial x_n}\right]
$$
\end{definition}

As it was remarked in the introduction, the first polar map $g^1$ coincides with the
classical projective gauss map. \\

The higher degree polar maps are only rational in general. Interestingly
however,  they can still be regular in the presence of certain singularities. More precisely, the  following is true:

\begin{theorem} If a hypersurface of degree $d$ has a regular polar map of degree $p$, it also
has regular polar maps of degree $q$ for $p \leq q \leq d$
\end{theorem}
\begin{proof}
This follows immediately from the Euler relation: If the degree $p$ polar map is regular, then for
every $\xi \in X$
some $p$-th partial derivative $\frac{\partial^p F}{\partial x_{\alpha}}(\xi)$ is not zero. On the
other hand,  if all
the $q$-th partial derivatives are zero at $\xi$ then lemma (~\ref{workhorse}) implies
that  $\frac{\partial^p F}{\partial x^{\alpha}}(\xi) = 0$  for  all $\alpha$ !
\end{proof}

\begin{example}
Let $V \subset \PP^2$ be the nodal plane cubic defined by 
$F(x_0,x_1,x_2)=x_2x_1^2-x_0^3-x_0^2x_2$, then $V$ does
not have a regular polar map of degree $1$, but it has a well defined
polar map of degree $2$ (see figure \ref{fig:nodalcubic}):
$$
\overline{\xi} \mapsto V_{\overline{\xi}}=V(-(3\xi_0+\xi_2)x_0^2-2\xi_0x_0x_2+\xi_2x_1^2+2\xi_1x_1x_2)
$$
\begin{figure}[h]
\centering
\includegraphics[scale=0.09]{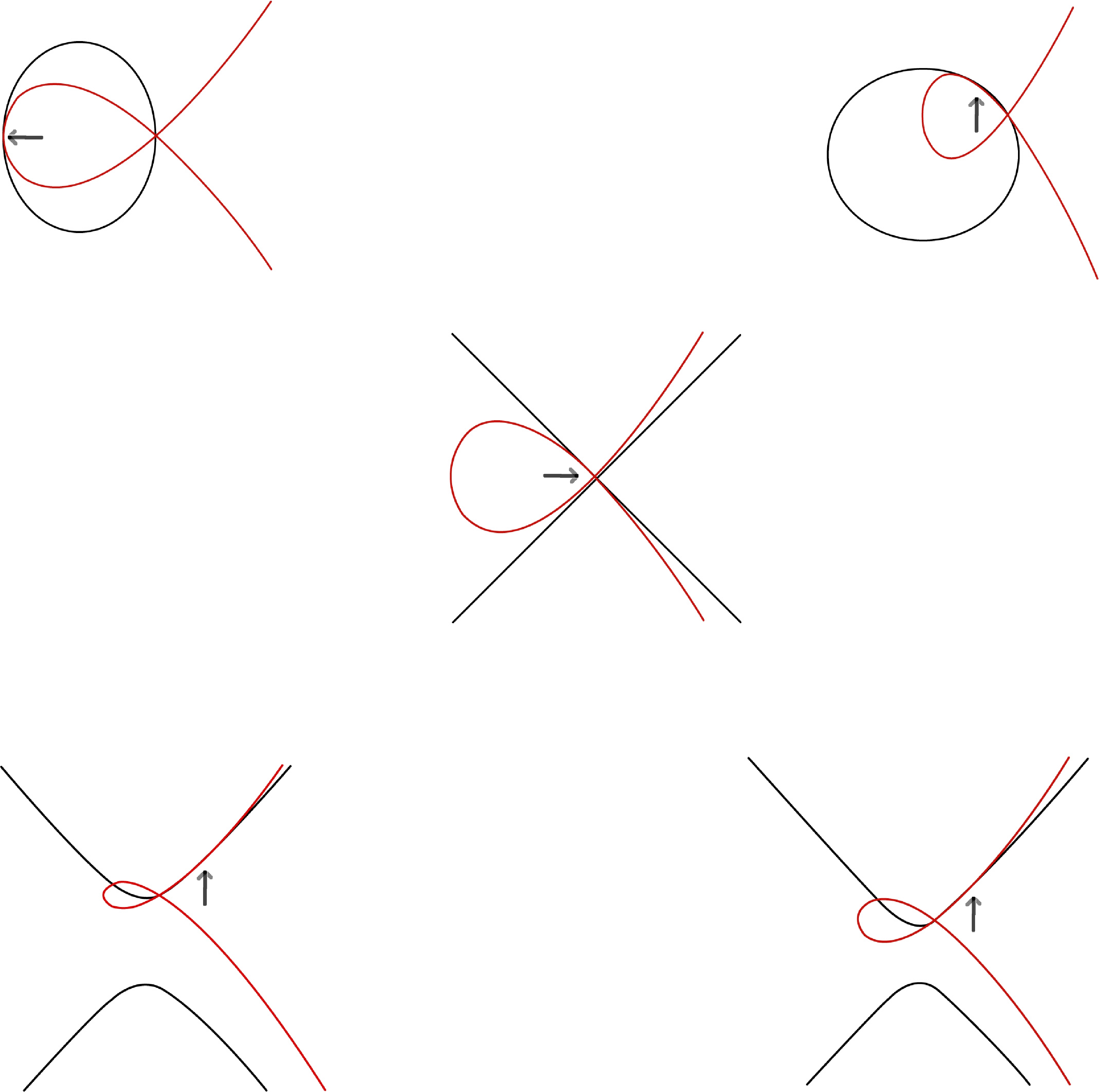}
\caption{Second degree approximations to the nodal cubic: The red curve is the nodal cubic,
the black curves are the conics approximating the curve at the point signaled by the arrow. }
\label{fig:nodalcubic}
\end{figure}

That is, to every point $v$ we associate a quadric which approximates the curve at $v$.
Notice that at the node $[0:0:1]$ we do have a well defined second
order approximation: the union of the two possible tangents.
\end{example}
\section{Degree and Dimension}
The closure of the image of the first polar map defines the dual variety of the hypersurface $X$.
We could ask what are the general properties of the images of these higher degree polar maps.
The following results extend some classic results of projective geometry. 

\begin{theorem}
\label{cone_char}
 If $X$ is not a cone, then  the $(d-1)$ polar map $g^{(d-1)}$is an isomorphism from
$X$ into its image. (This extends the well known result for the duals of smooth quadrics. Note
that any singular quadric is a cone.) 
\end{theorem}

In order to prove this theorem we will first prove a characterization of cones. Recall that 
a cone is the linear join $Y \# p$ of a variety $Y \in \P^n$ with a point $p \in \P^n$ such that
$p \notin Y$.

\begin{lemma} $X \subset \P^n$ is a cone if and only if  there is some $\xi \in X$ such that
${\rm mult}_{\xi}X=d$ (recall that $X$ is a hypersurface of degree $d$).
\end{lemma}
\begin{proof}
If $X$ is a cone, $X= Y \# p$ then $p$ is a point such that ${\rm mult}_pX= d$. Conversely, 
if $\xi$ is a point of  multiplicity $d$, let  $l$ be any line passing through $\xi$. 
If $l$ intersects $X$ at any other point $q$ then $\overline{\xi q}=l \subset X$ (otherwise 
the degree of $X$ would be greater than $d$).
\end{proof}

Now we can prove theorem \ref{cone_char}.
\begin{proof}
Notice that using Chow coordinates, $g^{(d-1)}$ is a linear map. That is, it is the map induced
by a linear map $\tilde{L}:\P^n \rightarrow \P^{{n+k \choose k} -1}$ by restriction to $X$. This
means that  $\tilde{L}$ itself is
induced by a linear map $L: \C^{n+1} \rightarrow \C^{n+k \choose k}$. If $g^{(d-1)}$ is not
an injection then $L$ is certainly not an injection. We will show that this leads to a 
contradiction. \\

If $L$ is not an injection then it has a non-zero kernel. Let $\bar{\xi}$ be a non-zero vector in
that kernel. This means that all the $(d-1)$-partial derivatives of $F$ vanish at $\bar{\xi}$. But
the generalized Euler relation implies then that $\xi \in X$. This is a contradiction
because we get a point $\xi \in X$ such that ${\rm mult}_{\xi}X=d$ (because all the $(d-1)-$partial
derivatives vanish at $\xi$), i.e. $X$ is a cone (by
the previous lemma). 

\end{proof}

\begin{lemma}\label{classcalculation} If the degree $p$ polar Map is regular, then 
$$(g^p)^*(O(1))= O_X(d-p)$$
\end{lemma}
\begin{proof}
First of all, notice that if $g^p$ is regular, then it extends to a regular map defined
on {\em all} of $\P^n$:

$$\tilde{g}^p:\P^n \rightarrow \P^N$$

We see this as follows. Note that the extension is given by the same coordinate functions. 
The reason for having this regular extension is the following: If there is some point
$\xi \in \P^n$ where all the coordinate functions vanish simultaneously 
lemma (~\ref{workhorse}) then implies that $F(\xi)=0$, i.e. $\xi \in X$, but this
would imply that $g^p$ is not regular! \\

Now, since the coordinate functions of $\tilde{g}^p$ can be interpreted as sections
of the bundle $(\tilde{g}^p)^*(O(1))$ and they are polynomials of degree $d-p$, we
get that the pullback of $O(1)$ under this map is $O(d-p)$. Since $g^p$ is just
the restriction of $\tilde{g}^p$  we get the lemma.

\end{proof}

\begin{theorem} \label{pullbackformula}
If the degree $p$ polar map is regular, the image of the $p$-th Gauss image
variety is $n-1$ and  the degree of the $p$-th Gauss image 
variety is $d(d-p)^{n-1}$. (This extends
the classical formulas for the degree of the dual of a smooth hypersurface and the fact that
the dual of a smooth hypersurface is a hypesurface).  
\end{theorem}
\begin{proof}
Using the previous lemma, this becomes just a chern class calculation. We will denote
with $H_{\P^s}$ the class of a hyperplane in $H^2(\P^s, \Z)$ and with $H_{\P^s}^k$ its $k$-fold cup
product. Recall that $\dim X = n-1$. Using the previous lemma we obtain:

$$ (g^p)^*(H_{\P^N})=(d-p)H_{\P^n}$$
therefore
$$ (g^p)^*(H_{\P^N}^{n-1})=(d-p)^{(n-1)}H^{(n-1)}_{\P^n}$$

If $\langle, \rangle$ denotes the Kronecker pairing, then the
following calculation proves the result:

\begin{eqnarray*}
\deg g^p(X) &= & \langle g^p(X), H_{\P^N}^{(n-1)} \rangle \\
            &= & \langle (\tilde{g}^p)_*(X),H_{\P^N}^{(n-1)} \rangle \\
            &= & \langle X, \tilde{g}^*(H_{\P^N}^{(n-1)}) \rangle \\
            &= & \langle X, (d-p)^{(n-1)} H_{\P^n}^{(n-1)} \rangle \\
            &= & d(d-p)^{(n-1)}
\end{eqnarray*}
\end{proof}

Lemma (~\ref{classcalculation})
also allows us to prove the following calculation.

\begin{theorem} 
\label{gaussclass}
If $X$ is a smooth hypersurface, the cohomology class defined by the $p$-th polar map
$[\ga^p] \in H^2(X)$ satisfies 
$$[\ga ^p]=\frac{d-p}{d-1}[\ga^1]=\frac{d-p}{d-1}c(N_X(-1))$$ 
More generally,
$$[\ga^p] = c_1(O_X(d-p))$$
\end{theorem}
\begin{proof}
The fundamental result of Lawson and Michelsohn in ~\cite{LM} implies that the homotopy
class $[\ga^p]$ coincides with $c_1((\ga^p)^*(O(1))$. Lemma (~\ref{classcalculation})
provides this last calculation:
$$c_1((\ga^p)^*(O(1)))=c_1(O_X(d-p))$$
The first claim is a consequence of the adjunction formula:
$$O_X(d)= [X]|_X= N_X$$
But we can also write
$$O_X(d)= O_X(d-1) \otimes O_X(1)=[g^1] \otimes O_X(1)$$
therefore
$$[g^1]=N_X(-1)$$
\end{proof}

\section{Examples and Applications}
The higher degree polar maps encode information about the underlying variety. For
example if the $p$-th polar map is regular then the variety cannot have singularities
of order $p$. The next theorem recovers the classical calculation of the number
of flexes of a smooth plane curve. Recall that a flex of a plain curve $C$ is a point
$p \in C$ where the tangent line has contact of order greater than $2$, that is,
the local intersection number at $p$ of the tangent line at $p$ and the curve is
greater than $2$.
\begin{theorem} Let $C$ be a smooth plane curve $C \subset \P^2$ of degree
$d \geq 2$. Then $C$ has $3d(d-2)$ flexes.
\end{theorem} 
\begin{proof}
Let $C$ be defined by a homogeneous polynomial $F$. A point $p$ is a flex if and only 
if the determinant of the Hessian matrix $H_p$ is zero, where
$$H_p =
\begin{pmatrix} 
\frac{\partial^2 F}{\partial x_0^2}(p) & 
          \frac{\partial^2 F}{\partial x_0x_1}(p) &
                      \frac{\partial^2 F}{\partial x_0x_2}(p) \\
\frac{\partial^2 F}{\partial x_0x_1}(p) & 
          \frac{\partial^2 F}{\partial x_1^2}(p) &
                      \frac{\partial^2 F}{\partial x_1x_2}(p) \\
\frac{\partial^2 F}{\partial x_0x_2}(p) & 
          \frac{\partial^2 F}{\partial x_1x_2}(p) &
                      \frac{\partial^2 F}{\partial x_2^2}(p)                       
\end{pmatrix}
$$
but $H_p$ is the quadratic form which defines the second polar map at $p$, i.e.
\begin{equation}\label{flex}
g^2(p) = \xi^T H_p \xi
\end{equation}
Notice that since $C$ is smooth, the second polar map
is regular. So the condition of $p$ being a flex is exactly the same
as $g^2(p)$ being a singular quadric. Now, singular quadrics form a hypersurface $\Delta$
of degree $3$ in the space $\P^5$ of all degree $2$ homogeneous polynomials
in three variables. This hypersurface $\Delta$ is given by the vanishing
of the determinant of the matrix defining a quadratic form.\\

Thus we are interested in computing the number of intersection points of
$g^2(C)$ with $\Delta$. But $\deg(g^2(C))=d(d-2)$ and $\deg(\Delta)=3$, therefore
Bezout's theorem implies that the number of flexes is $3d(d-2)$.  
\end{proof}

The next example shows how it is possible to have a hypersurface $X$ with
degenerate gauss map (the image of the gauss map will be a curve) and
nevertheless the second polar map has the same dimension as the hypersurface.

\begin{remark}
It would be interesting to find examples of hypersurfaces $X$ with degenerate higher
polar images such that $X$ is not a cone. If we define a hypersurface to be {\em 
$k$-defective} if its $k$-th polar map has a lower dimension than $X$ itself, then we
may ask for a charachterization of $k$-defective hypersurfaces ($1$-defective hypersurfaces
coincide with the notion of defective hypersurfaces, i.e. having a dual which is not
a hypersurface).
\end{remark}

\begin{example} Let $\Sigma \in \P^3$ be the rational normal curve. That is, 
$\Sigma$ is the image of the rational parametrization 
$$t \mapsto [1:t:t^2:t^3]$$
It is known that the dual variety $\Sigma^{\spcheck}$ is a hypersurface
defined by the discriminant of the general single variable polynomial of degree $3$
(cf ~\cite{GKZ} ch. 1.), namely, the equation defining the dual hypersurface
is:
$$\Delta = x_1^2x_2^2-4x_1^3x_3^2-27x_0^2x_3^2+18x_0x_1x_2x_3$$
Now, $\Sigma^{\spcheck}$ must necesarilly be singular, since otherwise
the dual variety would be a hypersurface. But the singularities of
$\Sigma^{\spcheck}$ actually have order $1$, therefore the second
polar map is regular and using our calculations we can conclude
that $g^2(\Sigma^{\spcheck})$ is a surface of degree $4(4-2)^2=16$ in
$\P^9$.

\end{example}

\end{document}